\documentclass[10pt]{article}
\usepackage{amsmath,amsthm,amssymb}

\newtheorem{theorem}{Theorem}[section]

\newtheorem{lemma}[theorem]{Lemma}

\renewcommand\ge\geqslant
\renewcommand\geq\geqslant
\renewcommand\le\leqslant
\renewcommand\leq\leqslant

\newcommand{\PP}{\mathbb{P}}

\newcommand{\RR}{\mathbb{R}}

\title{Discriminants of multilinear systems}
        
\author{Ioannis Z.~Emiris\footnote{
        Department of Informatics \& Telecommunications,
        National Kapodistrian University of Athens,
        Panepistimiopolis 15784, Greece. E-mail: {\sf emiris@di.uoa.gr}.},
\hspace*{1cm}
        Raimundas Vidunas\footnote{Graduate School of Information Science and Technology, Osaka University, Osaka, Japan.
        E-mail: {\sf rvidunas@gmail.com}.}}
\date{\today}

\begin{document}
\maketitle

\begin{abstract}
We study well-constrained
bilinear algebraic systems in order to formulate their discriminant.
We derive a new determinantal formula for the discriminant of a
multilinear system that appears in
the study of Nash equilibria of multiplayer games with mixed strategies.
\end{abstract}

\section{Introduction} 

We study well-constrained bilinear algebraic systems.
We aim at compact formulae for the discriminant of such systems so as
to improve the complexity of computing them.
One method would be for discriminants to be computed via implicitization
\cite{EmKaKoLBspm}.

In general, matrix formulae for the discriminant would be preferable but
they are quite hard to obtain and very few currently exist.
For instance, the discriminant of a single univarite polynomial is given,
up to a multiplicative monomial factor, by the determinant of the
Sylvester matrix of the polynomial and its derivative.
For a more general study see \cite{SturmfHur}.  

The lack of compact discriminant formulae is in contrast
to resultant matrices, which have been 
extensively studied and for which compact formulae exist for a large
number of system families. In particular,
the resultant matrices of overconstrained multihomogeneous systems
have been studied by Dickenstein, Mantzaflaris, and Emiris
\cite{DicEmi03,EmiManJsc} and, earlier,
by Sturmfels, Weyman, and Zelevinsky \cite{StZe,WeZe}.

\section{Purely bilinear systems}

Consider a bilinear polynomial system of $n+m$ equations on $\RR^n\times\RR^m$:
\begin{align} \label{eq:bilin}
F_k: \sum_{i=0}^n \sum_{j=0}^m a^{(k)}_{i,j} x_i y_j =0, \qquad 1\le k \le n+m.
\end{align}
The set of monomials appearing in each polynomial is generically 
$A=\{x_0,x_1,\ldots,x_n\}\times \{y_0,y_1,\ldots,y_m\}$.
Assuming the system is unmixed,
the generic number of solutions is the volume
(normalized to~1 for unit
simplex $\Delta_{n+m}$)
of the simplex product $\Delta_n\times\Delta_m$:
\begin{equation}
{n+m\choose n}=(n+m)! \left( \frac1{n!} \times \frac1{m!} \right).
\end{equation}
The {\em discriminant} $\Delta_A(F_1,\ldots,F_{n+m})$ of the system is
 the irreducible polynomial (with coprime coefficients, defined up to a sign)
 in the coefficients $a^{(k)}_{i,j}$ which vanishes whenever the system (\ref{eq:bilin})
 has a multiple solution. 

The discriminant can be computed (up to superflous factors) by eliminating all affine
variables except one,
and computing the discriminant of the univariate elimination polynomial.
For $n=m=1$, with $a_{ij}=a^{(1)}_{i,j}$, $b_{ij}=a^{(2)}_{i,j}$, we have
\begin{align} \label{eq:p11}
\Delta_A(F_1,F_2)= & \left(
\left| \begin{matrix}
a_{00} & a_{01} \\ b_{10} & b_{11}
\end{matrix} \right| 
- \left| \begin{matrix}
a_{10} & a_{11} \\ b_{00} & b_{01}
\end{matrix} \right|
\right) \left(
\left| \begin{matrix}
a_{00} & a_{10} \\ b_{01} & b_{11}
\end{matrix} \right| 
- \left| \begin{matrix}
a_{01} & a_{11} \\ b_{00} & b_{10}
\end{matrix} \right|
\right) \nonumber \\ 
& -4 \left| \begin{matrix}
a_{00} & a_{01} \\ a_{10} & a_{11}
\end{matrix} \right| 
\left| \begin{matrix}
b_{00} & b_{01} \\ b_{10} & b_{11}
\end{matrix} \right|.
\end{align}

\section{Degree bound}

This section bounds the degree of the discriminant.

For sparse polynomial systems, a degree bound was obtained by
Cattani, Cueto, Dickenstein, di Rocco and Sturmfels \cite{CCDRS}
and a more special case settled in \cite{DiEmKa}.

The discriminant equals the resultant of the equations (\ref{eq:bilin}) and $J=0$,
where
\begin{equation}
J=\det \left( 
\left. \frac{\partial F_i}{\partial x_j} \right|_{j=1\ldots n} \quad
\left. \frac{\partial F_i}{\partial y_j} \right|_{j=1\ldots m} 
\right)_{\!i=1\ldots n+m}
\end{equation}
is the Jacobian (determinant). The first $n$ columns of the Jacobian matrix
do not depend on the variables $x_j$ and are linear in the variables $y_j$,
while the last $m$ columns of the Jacobian matrix
do not depend on the variables $y_j$ and are linear in the variables $x_j$.
Therefore the Jacobian is (homogeneous) of degree $m$ in the $x_j$'s,
and of degree $n$ in the $y_j$'s.  The support is a product of scaled simplexes:
$m\Delta_n \times n\Delta_m$. The Jacobian is multilinear in the coefficients
of $a^{(k)}_{i,j}$, linear for each group with fixed $k$.

To bound the degree of the discriminant in the variables $a^{(1)}_{i,j}$, 
we compute
\begin{align*}
MV(J,F_2,\ldots,F_{n+m})+MV(F_1,F_2,\ldots,F_{n+m}) \deg_k J.
\end{align*}
The first term is (up to the factor $1/n!m!$) the permanent of a $(n+m)\times(n+m)$ matrix 
with $m$ $n$'s and $n$ $m$'s in one row, and all other entries equal to 1,
hence $2nm(n+m-1)!/n!m!$. The degree bound is 
\begin{align}
\left( \frac{2nm}{n+m}+1 \right) {n+m\choose n}.
\end{align}
The total degree is $n+m$ times larger.

The actual degrees appear to be smaller: 2 instead of 4 for $n=m=1$,
and 4 instead of 7 for $\{n,m\}=\{1,2\}$.

\section{Ideals containing the discriminant}

\begin{theorem}
The discriminant of the bilinear system $(\ref{eq:bilin})$
is in the ideal generated by the maximal minors of the $(m+n)(n+1)\times (m+1)$ matrix 
\begin{equation} \label{eq:ideal1}
(\partial F_i/\partial x_j)_{1\le i\le m+n,0\le j\le n},
\end{equation}
where each row represents a linear polynomial in the $y_k$'s,
and the columns correspond to the variables $(y_0:y_1:\cdots:y_m)$.
\end{theorem}
\begin{proof}
We have to prove that if the maximal minors vanish, the discriminant is zero. 
If the maximal minors vanish, we have a kernel $(u_0,u_1,\ldots,u_m)$
of the matrix in (\ref{eq:ideal1}). Hence the $(m+n)(n+1)$ derivatives $\partial F_i/\partial x_j$
vanish with all $y_i=u_i$. By Euler's relation
\begin{equation} \label{eq:euler}
F_i=\sum_{k=0}^n x_k \frac{\partial F_i}{\partial x_k}
\end{equation}
we conclude that all $F_i$ vanish at all $y_i=u_i$ and with any $x_i$.
The derivatives $\partial F_i/\partial x_j$ with $j\neq 0$ form the 
$(m+n)\times n$ zero submatrix of the Jacobian. Thus we have a whole subspace
of singular solutions of the system (\ref{eq:bilin}), hence the discriminant vanishes.
\end{proof}

Similarly, the discriminant must be in the ideal 
generated by the minors of the $(m+n)(m+1)\times (n+1)$ matrix 
\[
(\partial F_i/\partial y_j)_{1\le i\le m+n,0\le j\le m},
\] 
where the columns correspond to the variables $(x_0:x_1:\ldots:x_m)$.
In particular, the discriminant (\ref{eq:p11}) is in the minor ideals $I_1,I_2$ of
\begin{equation} \label{eq:r111}
\left( \begin{array}{cccccc}
a_{00} & a_{01} \\
a_{10} & a_{11} \\
b_{00} & b_{01} \\
b_{10} & b_{11} \\
\end{array} \right), \qquad
\left( \begin{array}{cccccc}
a_{00} & a_{10} \\
a_{01} & a_{11} \\
b_{00} & b_{10} \\
b_{01} & b_{11} \\
\end{array} \right).
\end{equation}
The discriminant is then in the radical of the product ideal of (intersecting) $I_1$ and $I_2$.
In this case, the discriminant is in the product ideal itself.

Other ideals are: {\em higher discriminant} ideals, characterizing the parameters
of the polynomial system with a multiple root of higher multiplicity,
or more than one multiple root.
Also, the ideal defining the singularity locus of the discriminant hypersurface.

\section{Sparse systems}

This section focuses on sparse multilinear systems, in particular when
each polynomial (or subset of polynomials)
does not depend on a subset of the variables.
These appear in the study of Nash equilibria in \cite{EmiVid14}.

Consider the bilinear system on $\PP^1\times\PP^1\times\PP^1$: 
\begin{eqnarray} \label{eq:p111}
H_1: & a_0x_1y_1+a_1x_1y_0+a_2x_0y_1+a_4x_0y_0 &=0, \nonumber \\
H_2: & b_0x_1z_1+b_1\,x_1z_0+b_3\,x_0z_1+b_4x_0z_0 &=0,\\
H_3: & c_0\,y_1z_1+c_2\,y_1z_0+c_3\,y_0z_1+c_4y_0z_0 &=0. \nonumber
\end{eqnarray}
The generic number of solutions equals 2.
The discriminant equals
\begin{align}
\Delta(H_1,H_2,H_3)=& \left( a_0 \left| \begin{matrix}
b_3 & b_4 \\ c_3 & c_4 \end{matrix} \right|
-a_1  \left| \begin{matrix}
b_3 & b_4 \\ c_0 & c_2 \end{matrix} \right|
-a_2  \left| \begin{matrix}
b_0 & b_1 \\ c_3 & c_4 \end{matrix} \right|
+a_4  \left| \begin{matrix}
b_0 & b_1 \\ c_0 & c_2 \end{matrix} \right|
\right)^{\!2} \nonumber \\
& -4\left| \begin{matrix}
a_0 & a_1 \\ a_2 & a_4 \end{matrix} \right|
\left| \begin{matrix}
b_0 & b_1 \\ b_3 & b_4 \end{matrix} \right|
\left| \begin{matrix}
c_0 & c_2 \\ c_3 & c_4 \end{matrix} \right|.
\end{align}
In the following theorem, the $6\times 6$ matrix is made up of column pairs
similar to the $4\times 2$ matrices in (\ref{eq:r111}):  
the first two columns encode all partial derivatives of $H_1,H_2,H_3$ 
that are linear in $x_1,x_0$, etc.
The columns of the $6\times 6$ matrix thereby
correspond to the multihomogeneous variables
$(x_1:x_0)$, $(y_1:y_0)$, $(z_1:z_0)$. 

\begin{theorem}
The discriminant $\Delta(H_1,H_2,H_3)$ equals 
\begin{equation} \label{eq:r111a}
\det  \left( \begin{array}{cccccc}
0 & 0 & a_0 & a_1 & b_0 & b_1 \\
0 & 0 & a_2 & a_4 & b_3 & b_4 \\
a_0 & a_2 & 0 & 0 & c_0 & c_2 \\
a_1 & a_4 & 0 & 0 & c_3 & c_4 \\
b_0 & b_3 & c_0 & c_3 & 0 & 0 \\
b_1 & b_4 & c_2 & c_4 & 0 & 0
\end{array} \right).
\end{equation}
\end{theorem}

\begin{proof}
For a conceptual proof, we relate a multiple root $(x_1:x_0),(y_1:y_0),(z_1:z_0)$ 
of the system (\ref{eq:p111}) and the kernel  $(\lambda_1:\lambda_2:\lambda_2)$ 
of the transposed Jacobian
\begin{equation} \label{eq:j111}
\left( \begin{array}{ccc}
\partial F_1/\partial x_1 & \partial F_2/\partial x_1 & 0 \\
\partial F_1/\partial y_1 & 0 & \partial F_3/\partial y_1 \\
0 & \partial F_2/\partial z_1 & \partial F_3/\partial z_1
\end{array} \right)
\end{equation}
to a kernel vector
$(u_1:u_2:u_3:u_4:u_5:u_6)$ of the matrix in (\ref{eq:r111a}), and vice versa.
The relation is as follows:
\begin{align}
(u_1:u_2:u_3:u_4:u_5:u_6)=& \left(
\frac{x_1}{\lambda_3}:\frac{x_0}{\lambda_3}:
\frac{y_1}{\lambda_2}:\frac{y_0}{\lambda_2}:
\frac{z_1}{\lambda_1}:\frac{z_0}{\lambda_1} \right), \nonumber\\
\left( \frac{x_1}{x_0}, \frac{y_1}{y_0},\frac{z_1}{z_0} \right) = &
\left( \frac{u_1}{u_2}, \frac{u_3}{u_4},\frac{u_5}{u_6} \right), \nonumber
\end{align}
The 1st, 3rd and 5th rows of (\ref{eq:r111a}) multiplied by the vector $(x_1,x_0,y_1,y_0,z_1,z_0)$ 
give the non-zero entries of the Jacobian matrix. The 2nd, 4th and 6th rows give the derivatives
with respect to $x_0,y_0,z_0$, This allows to complete three Euler identities like (\ref{eq:euler}),
and relate the kernel element with a singular root of (\ref{eq:p111}).
\end{proof}

We have the following observation.
\begin{lemma}
The system $(\ref{eq:p111})$ has a multiple root if and only if the quadratic form
$F+G+H$ (in the six variables $x_1,x_0,y_1,y_0,z_1,z_0$) degenerates.
\end{lemma}

The $2\times2$ blocks of (\ref{eq:r111a}) represent the following derivatives:
\begin{equation} \label{eq:rr111}
\left( \begin{array}{ccc}
0 & \partial F_1/\partial x & \partial F_2/\partial x \\
\partial F_1/\partial y & 0 & \partial F_3/\partial y \\
\partial F_2/\partial z & \partial F_3/\partial z & 0
\end{array} \right).
\end{equation}
The determinants of the matrices (\ref{eq:rr111}) and
(\ref{eq:j111}) match formally.

Direct generalization to $\PP^k\times\PP^\ell\times\PP^m$ is hardly possible
if the equation blocks have the sizes $k,\ell,m$, 
because the derivative blocks have non-matching number of columns. 
But we might assume the equation blocks to be of equal size,
and then the matrix is constructed correctly. But would its determinant indeed
be the system discriminant? 


\begin{thebibliography}{CCD{\etalchar{+}}14}

\bibitem[CCD{\etalchar{+}}14]{CCDRS}
E.\ Cattani, M.A.\ Cueto, A.\ Dickenstein, S.\~Di Rocco, and B.\ Sturmfels.
\newblock Mixed discriminants.
\newblock {\em Math.\ Z.}, 274:761--778, 2013.  

\bibitem[DE03]{DicEmi03}
A.\ Dickenstein and {I.Z.} Emiris.
\newblock Multihomogeneous resultant formulae by means of complexes.
\newblock {\em J.~Symbolic Computation}, 36(3-4):317--342, 2003.
\newblock Special issue on ISSAC 2002.

\bibitem[DEK14]{DiEmKa}
A.~Dickenstein, I.~Emiris, and A.~Karasoulou.
\newblock Plane mixed discriminants and toric jacobians.
\newblock In {\em SAGA: Advances in ShApes, Geometry, and Algebra}, volume~10
  of {\em Geometry and Computing}, pages 105--121. Springer, 2014.

\bibitem[EKK{\etalchar{+}}13]{EmKaKoLBspm}
I.Z.\ Emiris, T.\ Kalinka, C.\ Konaxis and T. Luu Ba,
Sparse implicitization by interpolation:
Characterizing non-exactness, and an application to computing discriminants,
{\em J. CAD, Spec.\ Issue on Solid \& Physical Modeling},
45:252--261, 2013, DOI 10.1016/j.cad.2012.10.008.

\bibitem[EM12]{EmiManJsc}
{I.Z.} Emiris and A.~Mantzaflaris.
\newblock Multihomogeneous resultant matrices for systems with scaled support.
\newblock {\em J.~Symbolic Computation}, 47:820--842, 2012.
\newblock Special Issue on ISSAC 2009.

\bibitem[EV14]{EmiVid14}
{I.Z.} Emiris and R.~Vidunas.
\newblock Root counts of semi-mixed systems, and an application to counting
  {Nash} equilibria.
\newblock In {\em Proc.\ Annual ACM Intern.\ Symp.\ on Symbolic and Algebraic
  Computation (ISSAC)}, pages 154--161, Kobe, Japan, 2014. ACM Press.

\bibitem[Stu16]{SturmfHur}
B.~Sturmfels.
\newblock The {Hurwitz} form of a projective variety.
Presented in MEGA 2015.
\newblock Submitted to {\em J. Symbolic Computation}, 2016.

\bibitem[SZ94]{StZe}
B.\ Sturmfels and A.~Zelevinsky.
\newblock Multigraded resultants of {Sylvester} type.
\newblock {\em J.~Algebra}, 163(1):115--127, 1994.

\bibitem[WZ94]{WeZe}
J.~Weyman and A.~Zelevinsky.
\newblock Multigraded formulae for multigraded resultants.
\newblock {\em J.~Algebr.\ Geom.}, 3(4):569--597, 1994.

\end{thebibliography}

\newcommand{\etalchar}[1]{$^{#1}$}

\end{document}